\documentclass{amsart}

\newtheorem{theorem}{Theorem}[section]

\theoremstyle{definition}

\newtheorem{corollary}[theorem]{Corollary}

\theoremstyle{remark}
\newtheorem{remark}[theorem]{Remark}

\numberwithin{equation}{section}

\begin{document}

\title{Sharp Markov-type Inequalities for Rational Functions on Several Intervals}

\author{M. A. Akturk}
\address{Department of Engineering Sciences, Istanbul University, 34320 Istanbul, Turkey}
\curraddr{Department of Engineering Sciences, Istanbul University, 34320 Istanbul, Turkey}
\email{mehmetaliakturk@yandex.com}
\thanks{The authors were supported by Scientific and Technological Research Council of Turkey (TUB\.{I}TAK), joint project with Russian Foundation of Basic Research (RFBR), no.113F369.}

\author{A. Lukashov}
\address{Department of Mathematics, Fatih University, 34500 Istanbul, Turkey \\
              and Department of Mechanics and Mathematics, Saratov State University, 410012 Saratov, Russia}
\email{alukashov@fatih.edu.tr}

\subjclass[2010]{41A17, 41A20}

\date{July 11, 2014 and, in revised form, Month Day, Year.}


\keywords{Inequalities in approximation, Approximation by rational functions}

\begin{abstract}
In this paper, we give sharp Rusak- and Markov-type inequalities for rational functions on several intervals when the system of intervals is a \textquotedblleft rational function inverse image\textquotedblright\, of an interval and those functions are large in gaps.
\end{abstract}

\maketitle

\section{Introduction}
In 1889 A. A. Markov proved that if $P_{n}(x)$ is a polynomial of degree $n$ and $\left\vert P_{n}(x)\right\vert \leq M$ on an interval of length $2L,$ then  $\left\vert P^{\prime}_{n}(x)\right\vert \leq n^2 ML$ on the same interval and factor $n^2$ is attained for (transformed) Chebyshev polynomials of the first kind. Reader may consult papers \cite{Boas1978,Shadrin2004} and books \cite{Borwein1995,Rahman2002} for more detailed story of this inequality and its generalizations.

It is clear that Markov inequality can not be generalized for the set of all rational functions of degree $n,$ although there are some Markov-type inequalities for rational functions with free poles \cite{Danchenko1996,Dolzhenko1963,Gonchar1961}. For rational functions with fixed denominators under additional restrictions on the poles sharp Markov-type inequalities on an interval were obtained by V.N. Rusak \cite{Rusak1979}.

Generalizations of Markov inequality onto polynomials and rational function for sets more complicated then one interval were considered in many papers \cite{Erdélyi2000,Plesniak2005,Privalov1983,Totik2002}, but to the best of our knowledge the sharp factor for Markov inequalities is not known even for the simplest case of polynomials of even degree on two symmetric intervals. Asymptotically sharp Markov-type inequalities for polynomials on several intervals were obtained in \cite{Borwein1981,Totik2011}. We start with one result of V.N. Rusak which was the base for sharp Markov type inequalities on one interval for rational functions with fixed poles.

\begin{theorem} \cite{Rusak1979}
If algebraic fraction of the from
\begin{eqnarray*}
r_{n}(x)=\frac{p_{n}(x)}{\sqrt{t_{2n}(x)}},
\end{eqnarray*}%
\begin{equation}
t_{2n}(x)=\prod\limits_{k=1}^{2n}\left( 1+a_{k}x\right) ,~~~-1\leq x\leq 1,
\end{equation}%
where $p_{n}(x)$ is an algebraic polynomial of degree at most $n$ with
complex (real) coefficients, numbers $a_{k}$ are either real or pairwise complex conjugate, $\left\vert
a_{k}\right\vert <1, k=1,\ldots,2n,$ satisfies the condition
\begin{equation}
\left\vert r_{n}(x)\right\vert \leq 1,~~~-1\leq x\leq 1,
\end{equation}
then the estimate
\begin{equation} \label{rus27}
\left\vert r_{n}^{\prime }(x)\right\vert \leq \left\{
\begin{array}{c}
\frac{\lambda _{n}(x)}{\sqrt{1-x^{2}}},~x_{1}\leq x\leq x_{n}, \\
\left\vert m_{n}^{\prime }(x)\right\vert ,~~~-1\leq x\leq x_{1},~x_{n}\leq x\leq 1,
\end{array}%
\right.
\end{equation}
is valid. Here $\left\{ x_{k}\right\} ,$ $-1<x_{1}<\ldots<x_{n}<1$ are zeros
of the cosine fraction $m_{n}(x)=\cos\frac{1}{2}\sum\limits_{k=1}^{2n} \arccos \left(\frac{x+a_{k}}{1+a_{k}x}\right)$, and $\lambda_n (x)= \frac{1}{2}\sum\limits_{k=1}^{2n}\frac{\sqrt{1-a_{k}^{2}}}{1+a_{k}x}.$ Equality sign in (\ref{rus27}) is attained only
for functions $r_{n}(x)\equiv \varepsilon m_{n}(x),~\left\vert \varepsilon \right\vert =1.$
\end{theorem}
 The main goal of the paper is to give maximal  (in a sense) possible extension of Rusak's inequality for rational functions on several intervals. As a consequence we give sharp Markov-type inequality for rational functions with an additional restriction.
Note that the equality for our Markov-type inequality is attained for rational functions which can be considered as an analogue of Chebyshev-Markov rational functions (their numerators are called also Bernstein-Szeg\H{o} orthogonal polynomials) on several intervals (see, for example \cite{Lebedev2003,Lukashov1998,Lukashov2004,Peherstorfer1993}). We would like to remark that general result of V. Totik \cite{Totik2011} contains $o(1)$ as $n\longrightarrow \infty,$ hence it says nothing for a fixed (maybe small) $n$.

\section{Results}

Let $\Re \left( \xi _{1},\ldots,\xi _{2n}\right)$ be the set of all \textquotedblleft\ rational functions \textquotedblright  of the form
\begin{equation} \label{arr}
r(x)=\frac{b_{0}x^{n}+b_{1}x^{n-1}+\ldots+b_{n}}{\sqrt{\rho _{\nu }(x)}},
\end{equation}%
$b_{0},\ldots,b_{n}\in\mathbb{C}$ and $\rho _{\nu }(x)=\prod\limits_{j=1}^{2n}\left( x-\xi _{j}\right)$ is a real polynomial of degree $\nu $ which is positive on $E=\bigcup\limits_{j=1}^{l}\left[ a_{2j-1},a_{2j}\right],$ $-1=a_{1}<a_{2}<\ldots<a_{2l}=1$ ($\xi _{j}$ might be equal to $\infty,$ then \: $\left( x-\xi _{j}\right)$ should be omitted)
and
\begin{eqnarray*}
\varpi _{E}\left( z,x\right) =\frac{\partial }{\partial x}\omega\left(
z,E\cap \left[ a_{1},x\right] ,\mathbb{C}\backslash E\right) ;\omega_{k}(\xi _{j})=\omega\left( \xi_{j},\left[ a_{2k-1},a_{2k}\right],\mathbb{C}\backslash E\right), k=1,\ldots,l.
\end{eqnarray*}
Here $\omega\left( z,G,\mathbb{C}\backslash E\right)$ is the harmonic measure of a set $G\subset E$ at a point $z\in\mathbb{C}\backslash E.$

Consider also a subclass $\Re^{*} \left( \xi _{1},\ldots,\xi _{2n}\right)$ of $\Re \left( \xi _{1},\ldots,\xi _{2n}\right)$ which consists of those functions $r \in \Re \left( \xi _{1},\ldots,\xi _{2n}\right)$, which satisfy $\left\vert r(x)\right\vert > \left\Vert r \right\Vert_{C(E)}$ for all $x \in [-1,1] \backslash E.$

\begin{theorem} \label{thmar4}
Suppose that $\sum\limits_{j=1}^{2n} \omega_{k}(\xi _{j})=2q_{k},~q_{k}\in
\mathbb{N},~k=1,\ldots,l,$ and $ \left\vert\xi _{j}\right\vert>1, j=1,\ldots,2n.$
Then for any $r\in \Re^{*} \left( \xi _{1},\ldots,\xi _{2n}\right), \left\Vert r \right\Vert_{C(E)}=1$ the inequality
\begin{equation} \label{special}
\left\vert r^{\prime }(x)\right\vert \leq \left\{
\begin{array}{c}
\gamma _{n}^{\prime }\left(x\right), x \in \widetilde{E}_{n}, \\
\left\vert m_{n}^{\prime }(x)\right\vert , x \in E \backslash \widetilde{E}_{n}
\end{array}%
\right.
\end{equation}
is valid, where
\begin{eqnarray*}
m_{n}(x) &=&\cos \left( \gamma _{n}\left( x\right) \right), \\
\gamma _{n}\left( x\right)&=&\frac{\pi }{2}\int\limits_{a_{1}}^{x}\sum\limits_{j=1}^{2n}\varpi _{E}\left( x,\xi _{j}\right) dx,\\
\widetilde{E}_{n} &=&\left[ x_{1},x_{q_{1}}\right] \cup \left[
x_{q_{1}},x_{q_{1}+q_{2}}\right] \cup \ldots\cup \left[x_{q_{1}+\ldots+q_{l-1}},x_{n}\right],
\end{eqnarray*}%
and $x_{1}<\ldots<x_{n}$ are zeros of $m_{n}$ (there are $q_{k}$ zeros on $\left[ a_{2k-1},a_{2k}\right], k=1,\ldots,l).$ 

For $r(x)\equiv \varepsilon m_{n}(x),~\left\vert \varepsilon \right\vert =1,$ inequality in (\ref{special}) is attained.

\end{theorem}

\begin{proof}
It is sufficient (compare \cite[proof of Corollary 5.1.5]{Borwein1995}) to prove Theorem \ref{thmar4} only for the case when numerator $p_{n}(x)$
has real coefficients.
Firstly consider the case $x \in \widetilde{E}_{n}.$
Then by \cite[Theorem 5]{Lukashov2004}
\begin{equation}
\left\vert r^{\prime }\left( x\right) \right\vert  \leq\sqrt{1-r^{2}\left(
x\right) } \gamma _{n}^{\prime }\left( x\right) \leq \gamma _{n}^{\prime }\left( x\right)  
\end{equation}

Next we consider the case $x\in E\backslash \widetilde{E}_{n}.$ Here we adopt Rusak's method \cite[p. 58-60]{Rusak1979}.

Let $\left\{ y_{j}\right\} _{j=1}^{n+l}$ be zeros of the sine
Chebyshev-Markov fraction $\frac{\sqrt{-H(x)}v_{n-l}(x)}{\sqrt{\rho _{\nu }(x)}},$ where $v_{n-l}(x)$ is determined from
\begin{equation*}
m_{n}^{2}(x)-\frac{H(x)v_{n-l}^{2}(x)}{\rho _{\nu }(x)}=1,
\end{equation*}%
and $H(x)=\prod\limits_{j=1}^{2l}\left( x-a_{j}\right) $ (see \cite[Theorem 3]{Lukashov2004}). Note that $m_{n}$ is a function of the form (\ref{arr}) and
\begin{eqnarray} \label{mar11}
m_{n}(y_{k})&=&(-1)^{\mu +j},\quad y_k\in[a_{2j-1},a_{2j}],\quad k=1,\ldots n+l \\
k&=&\sum\limits_{i=1}^{j-1}q_{i}+\mu+j-1,~~~  1\leq \mu \leq q_{j}+1,~ ~~~ j=1,\ldots,l.
\end{eqnarray}
If $r(x)$ is different from $m_{n}(x)$ then for $0<\lambda <1$ the difference $%
m_{n}(x)-\lambda $ $r(x)$ takes the same signs as $m_{n}$ at points $%
\left\{ y_{j}\right\} _{j=1}^{n+l}.$ Hence it has simple zeros on the set $E$
and can be written as
\begin{equation*}
m_{n}(x)-\lambda r(x)=C\frac{\prod\limits_{k=1}^{n}(x-z_{k})}{\sqrt{%
\prod\limits_{j=1}^{2n}\left( x-\xi _{j}\right) }}.
\end{equation*}%
Next, the derivative of this difference is equal to
\begin{equation} \label{mar12}
m_{n}^{\prime }(x)-\lambda r^{\prime }(x)=\frac{C}{2}\left(
m_{n}(x)-\lambda r(x)\right) \left[ 2\sum\limits_{k=1}^{n}\frac{1}{%
x-z_{k}}-\sum\limits_{j=1}^{2n}\frac{1}{x-\xi _{j}}\right].
\end{equation}%
Zeros of the difference $m_{n}^{\prime }(x)-\lambda r^{\prime }(x)$
coincide with zeros of square brackets in (\ref{mar12}), and they are zeros of the
logarithmic derivative of the rational function $R_{2n}(x)={\prod%
\limits_{k=1}^{n} \left( x-z_{k}\right) ^{2}}/{\rho _{\nu }(x)}.$

The condition $\left\vert \xi _{j}\right\vert >1,~j=1,\ldots,2n$ guarantees
that all poles of the rational function $R_{2n}(z)$ are contained in region $%
\left\vert z\right\vert \geq r>1,$ all zeros $\left\{ z_{k}\right\}$ of the function $\left\{R_{2n}(z)\right\} $ lie in the interval $(-1,1)\supset E.$

Bocher-Walsch Theorem \cite[Corollary 1 on p. 99]{Walsh1950} implies that $R_{2n}^{\prime }(z)$ has $n-1$ zeros in $\left\vert z\right\vert \leq 1$ and
other zeros are in $\left\vert z\right\vert \geq r>1.$ In the same way
because of (\ref{mar12}) the zeros of the difference $m_{n}^{\prime }(x)-\lambda
r^{\prime }(x)$ are located. So the function $m_{n}^{\prime }(x)-\lambda
r^{\prime }(x),~0<\lambda <1$ has no more than $n-1$ zeros in $(-1,1)$
and other zeros are in $\left\vert z\right\vert \geq r>1.$

By the continuity as $\lambda $ tends to $1$ we obtain the function $%
m_{n}^{\prime }(x)-r^{\prime }(x)$ has at most $n-1$ zeros on $\left[-1,1\right].$

From the identity%
\begin{equation*}
\left( \frac{m_{n}^{\prime }(x)}{\gamma _{n}^{\prime }(x)}\right)
^{2}+m_{n}^{2}(x)=1,
\end{equation*}%
and the inequality (with $\leq $sign it was given in \cite[Theorem 5]{Lukashov2004}, for $r\neq
\pm m_{n}$ strict inequality was obtained in \cite{Kalmykov2009})
\begin{equation*}
\frac{\left( r^{\prime }(x)\right) ^{2}}{\gamma _{n}^{\prime }(x)}%
+r^{2}(x)<1,
\end{equation*}%
we get%
\begin{equation*}
\left\vert r^{\prime }(x_{k})\right\vert < \left\vert m_{n}^{\prime }(x_{k})\right\vert.
\end{equation*}%

Since
\begin{equation*}
m_{n}^{\prime }(x) =-\sin \gamma _{2n}(x)\,\gamma _{n}^{\prime }(x),
\end{equation*}%
has alternate sign at points $\left\{ x_{k}\right\} _{k=1}^{n},$
the relation $sign\left\{ m_{n}^{\prime }(x_{k})-r^{\prime }(x_{k})\right\}
=\left( -1\right) ^{k}$, $k=1,\ldots,n$ holds. On the interval $\left(
x_{1},x_{q_{1}}\right) $ there are $q_{1}-1$ zeros of $m_{n}^{\prime }(x)-r^{\prime }(x),$ on $\left( x_{q_{1}+1},x_{q_{1}+q_{2}}\right) $ there
are $q_{2}-1$ zeros, and so on, therefore $n-l$ zeros of
the difference $m_{n}^{\prime }(x)-r^{\prime }(x)$ are on the set $\widetilde{E}_{n}.$
Since $m_{n}(x)-\lambda r(x)$ has the same value at $a_{2k}$ and $a_{2k+1}, k=1,\ldots,l-1,$ there is at least one zero of $m_{n}^{\prime }(x)-\lambda r^{\prime }(x)$ on $\left( a_{2k},a_{2k+1}\right), k=1,\ldots,l-1. $
So by the continuity as $\lambda$ tends to $1$ we obtain that $m_{n}^{\prime }(x)- r^{\prime }(x)$
 preserves sign inside intervals from $E\backslash \widetilde{E}_{n}.$

In particular it follows from $m_{n}^{\prime }(a_{1})-r^{\prime
}(a_{1})\leq0$ that $m_{n}^{\prime }(x)-r^{\prime }(x)\leq 0$ for $a_{1}\leq
x\leq x_{1}.$ Analogous consideration is valid for the algebraic fraction $%
-r(x).$ Therefore we obtain inequality $m_{n}^{\prime }(x)+r^{\prime
}(x)\leq0$ for $a_{1}\leq x\leq x_{1},$ and it gives the estimate $\left\vert
r^{\prime }(x)\right\vert \leq\left\vert m_{n}^{\prime }(x)\right\vert
,~a_{1}\leq x\leq x_{1}.$ Analogously other intervals from $E\backslash \widetilde{E}_{n}$ are considered.

\end{proof}

\begin{theorem} \label{thmar5}
Suppose that $\sum\limits_{j=1}^{2n} \omega_{k}(\xi _{j})=2q_{k},~q_{k}\in
\mathbb{N},~k=1,\ldots,l,$ and $\xi _{j} \in \overline{\mathbb{R}}, \left\vert\xi _{j}\right\vert>1, j=1,\ldots,2n.$
Then for any $r\in \Re^{*} \left( \xi _{1},\ldots,\xi _{2n}\right), \left\Vert r \right\Vert_{C(E)}=1$ the inequality
\begin{equation}
\left\Vert r^{\prime} \right\Vert_{C(E)} \leq \left\Vert m_{n}^{\prime}\right\Vert_{C(E)}
\end{equation}
holds.

\end{theorem}

\begin{proof}

Firstly consider the case $x \in \widetilde{E}_{n}.$
Then by Theorem \ref{thmar4}
\begin{eqnarray*}
\left\vert r^{\prime }\left( x\right) \right\vert  &\leq&
 \gamma _{n}^{\prime }\left( x\right)   \\
&\leq &\underset{x\in \widetilde{E}_{n}}{\max } \gamma
_{n}^{\prime }\left( x\right).
\end{eqnarray*}
By \cite{Benko2012} $\gamma _{n}^{\prime }\left( x\right)$ is convex hence
\begin{equation*}
\underset{x\in \widetilde{E}_{n} }{\max }\left( \gamma
_{n}^{\prime }\left( x\right) \right) =\underset{k=1,\ldots,n}{\max }\left( \gamma _{n}^{\prime }\left(
x_{k}\right) \right)
\end{equation*}
and we get by
\begin{equation*}
 \gamma _{n}^{\prime }\left(x_{k}\right)=\left\vert m_{n}^{\prime }\left(x_{k}\right) \right\vert, k=1,\ldots,n,
\end{equation*}
 the desired result,
\begin{equation*}
\left\vert r^{\prime}(x)\right\vert \leq \left\Vert m_{n}^{\prime}\right\Vert_{C(E)}, x\in \widetilde{E}_{n},
\end{equation*}
and application of second inequality in  Theorem \ref{thmar4} finishes the proof.

\end{proof}

\begin{corollary} \label{cor}
Let $\Pi _{n}$ denote the set of algebraic polynomials with real coefficients of degree at most $n$.
If $0<a<b,$ n is even and  $p_{n}\in \Pi^{*} _{n}$ where   $\Pi^{*} _{n}$ is the subclass of $\Pi _{n}$ which consists of polynomials $p_{n}\in \Pi _{n}$ such that $ p_n(x)> \left\Vert p_{n}\right\Vert _{\left[
-b,-a\right] \cup \left[ a,b\right] }$ for $ x\in (-a,a), $ then the inequality
\begin{equation*}
\left\Vert p_{n}^{\prime }\right\Vert _{\left[ -b,-a\right] \cup \left[ a,b%
\right] }\leq \frac{n^{2}b}{b^{2}-a^{2}}\left\Vert p_{n}\right\Vert _{\left[
-b,-a\right] \cup \left[ a,b\right] }.
\end{equation*}
is valid. It is attained for polynomials
\begin{equation*}
S_{n}(x)=T_{n/2}\left( \frac{2x^2-b^2-a^2}{b^2-a^2}\right),
\end{equation*}
where $T_{n}$ is the $n^{th}$ Chebyshev polynomial $\left( T_{n}=\cos(n\arccos(x))\right).$
\end{corollary}

\begin{proof}
It is sufficient to note that for $E=\left[ -b,-a\right] \cup \left[ a,b\right].$
\begin{equation*}
\varpi _{E}\left( x,\infty\right)=\frac{1}{\pi}\frac{\left\vert x\right\vert}{\sqrt{(b^2-x^2)(x^2-a^2)}},
\end{equation*}
and for $\xi _{1}=\ldots=\xi _{2n}=\infty, n$ is even, $m_{n}(x)=S_{n}(x).$

\end{proof}

\begin{remark}
Corollary \ref{cor} complements \cite[Theorem 2]{Borwein1981}.
\end{remark}

\begin{remark}
The condition $r\in \Re^{*} \left( \xi _{1},\ldots,\xi _{2n}\right)$ is essential for validity of Theorem \ref{thmar4}. Indeed, let $E=\left[ -1,-a\right] \cup \left[ a,1\right], \rho_{ \nu }(x)=1 $.
Here
\begin{equation*}
m_{4}(x)=\frac{8x^4- 8x^2 (1+a^2)+1+6a^2+a^4}{(1-a^2)^2}
\end{equation*}
and
\begin{equation*}
m_{4}^{\prime }(a)=-\frac{16a}{1-a^2}.
\end{equation*}
But for the Chebyshev polynomial $T_{3}(x)=4x^3-3x$ we have
\begin{equation*}
T_{3}^{\prime }(a)=12a^2-3,
\end{equation*}
hence
\begin{equation*}
\left\vert T_{3}^{\prime }(a)\right\vert > \left\vert m_{4}^{\prime }(a) \right\vert
\end{equation*}
for sufficiently small $a$. We don't know if the inequality in Theorem \ref{thmar5} is valid for $r\in \Re \left( \xi _{1},\ldots,\xi _{2n}\right)\backslash\Re^{*} \left( \xi _{1},\ldots,\xi _{2n}\right) $ or not.
\end{remark}

\begin{remark}
The main result of \cite{Akturk2014} which corresponds to Theorem 2.1, is valid under additional supposition
$\left\vert r(x)\right\vert >\left\Vert r \right\Vert_{C(E)}$ for all $x \in conv(E) \backslash E$.
\end{remark}

\begin{remark}
The supposition $\left(\left\vert \xi _{j}\right\vert>1\right)$ is essential even for the case of one interval, as the example (taken from \cite{Rusak1979})
$r\left( x\right)=\frac{72x^2-5x}{100x^2+1}$ shows (here $\left\Vert r^{\prime }\right\Vert _{\left[ -1,1\right]}\geq r^{\prime }(1)>\left\Vert m_{2}^{\prime }\right\Vert _{\left[ -1,1\right]}, m_{2}(x)=\frac{102x^2-1}{100x^2+1}, \left\Vert r\right\Vert _{\left[ -1,1\right]}\leq 1).$

\end{remark}

\paragraph{{\bf}ACKNOWLEDGEMENTS}
The authors are deeply grateful to reviewers and to Professor V. Totik for careful reading of the manuscript and remarks which helped to improve the text and correct mistakes.

\bibliographystyle{amsplain}

\end{document}